\documentclass[3p,times]{article}

\usepackage{amsmath}
\usepackage{amsthm}
\usepackage{booktabs}
\newtheorem{prop}{Proposition}

\newtheorem{defn}{Definition}

\newtheorem{propert}{Properties}
\newtheorem{cor}{Corollary}
\newtheorem{exmp}{Example}
\newtheorem{obs}{Observation}
\newtheorem{rem}{Remark}

\usepackage{amssymb}

\usepackage[figuresright]{rotating}
\usepackage{tcolorbox}

\usepackage[font=footnotesize,labelfont=bf]{caption}
\usepackage[font=footnotesize,labelfont=bf]{subcaption}

\usepackage[english]{babel}
\title{Monomiality and a New Family of Hermite Polynomials}

\author{Giuseppe Dattoli, Silvia Licciardi\footnote{Corresponding author: silviakant@gmail.com, silvia.licciardi@enea.it, orcid 0000-0003-4564-8866, tel. nr: +39 06-94005421.},\\[1.3ex]
	ENEA - Frascati Research Center, Via Enrico Fermi 45, 00044, \\Frascati, Rome, Italy  
}

\date{\today}
\begin{document}
	
	\maketitle

	\begin{abstract}
In this article we go deeply into the formulation and meaning of the monomiality principle and employ it to study the properties of a set of polynomials, which, asymptotically, reduce to the ordinary two variable Kampè-dè-Fèrièt family.  We derive the relevant differential equations and discuss the associated orthogonality properties, along with the relevant generalized forms.
\end{abstract}


	\section{Introduction}

	The Hermite polynomials belong to the Appèll family \cite{Appel} and the relevant properties can be conveniently framed within the context of the monomiality principle \cite{Bell,DG}. This is a modern formulation of a point of view, tracing back to Steffensen \cite{Steffensen2,Steffensen3,Steffensen}, but even to older researches by Jeffery (for a recent account see Ref. \cite{Jeffrey}), Boole \cite{Boole} and to other speculations developed almost two hundred years ago. These researches deepened their roots into the calculus of differences \cite{Jordan}, the first to be recognized as amenable for a symbolic interpretation. The rules underlying monomiality are fairly simple and can be formulated as reported below \cite{Bell,DG,Germano,Laguerre,Laguerre2} 
	
	\begin{propert}
		$\forall x\in\mathbb{R}, \forall n\in\mathbb{N}$, If a couple of operators $\hat{P},\hat{M}$ are such that:
	\begin{itemize}
		\item [a)] they do exist along with a differential realization,
        \item [b)] they can be embedded to form a Weyl algebra \cite{Babusci,Gallardo,Ottaviani}, namely if commutator is such that $[\hat{P},\hat{M}]=\hat{1}$,
        \item [c)] it is possible to univocally define a polynomial set such that:
	     \begin{equation}\label{p0}
	      i) \;p_0(x)=1, \qquad\qquad ii)\;\hat{P} p_0(x)=0, \qquad\qquad iii)\;p_n(x)=\hat{M}^n 1,
	     \end{equation}
	     	\end{itemize}
	then it follows that
		\begin{itemize}
		\item [d)]
	\begin{equation}\label{pn}
	\hat{M} p_n(x)=\hat{M}^{n+1}1=p_{n+1}(x),
	\end{equation}
		\item [e)] 	\begin{equation}\label{pn2}
		 \hat{P}p_n(x)=\hat{P}\hat{M}^n 1=n p_{n-1}(x)
		\end{equation}
		\end{itemize}
	\noindent and the polynomials $p_n(x)$ are said \textit{Quasi-Monomials}.
	\end{propert}
	\begin{proof}
	 Eq. \eqref{pn2} needs few lines of comment. We rearrange the operator product  $\hat{P}\hat{M}^n$ as\footnote{We remind that $[\hat{P},\hat{M}]=1\Rightarrow \hat{P}\hat{M}-\hat{M}\hat{P}=1$}  (see \cite{Louisell})
	\begin{equation}\label{key}
	\begin{split}
	\hat{P}\hat{M}\hat{M}^{n-1}&=(\hat{M}\hat{P}+1)\hat{M}^{n-1}=\hat{M}\hat{P}\hat{M}^{n-1}+\hat{M}^{n-1}=\\
	& =\hat{M}^2\hat{P}\hat{M}^{n-2}+2\hat{M}^{n-1}=\dots= \hat{M}^n\hat{P}+n\hat{M}^{n-1},
		\end{split}
	\end{equation}
	which eventually yields
	\begin{equation}\label{key}
\hat{P}\hat{M}^n 1=	\left(  \hat{M}^n\hat{P}+n\hat{M}^{n-1}\right) 1= \hat{M}^n\hat{P}\;1+n\hat{M}^{n-1}1.
	\end{equation}
	Being $\hat{M}^n\hat{P}\;1=0$ as a consequence of the $ii)$ of Eqs. \eqref{p0}, and using property $iii)$ too, we state the correctness of Eq. \eqref{pn2}.
	\end{proof}
	
	\begin{rem}
	The important point we like to convey is that the essence of the discussion on Monomiality is the existence of the operators $\hat{M}$ (multiplicative), which univocally define the set of polynomials $p_n(x)$ (not vice-versa), and $\hat{P}$ acting on the polynomials  as an ordinary derivative.
	\end{rem}
	
	\noindent According to the above statement polynomial set like Appèll, Sheffer \cite{Sheffer}, Boas Buck \cite{Boas}…can be ascribed to the monomial family, while others like e. g. Legendre, Chebyshev, Jacobi…\cite{L.C.Andrews,Abramovitz} are not be framed within such a context.
\\
	
	After these remarks, aimed at clarifying the frame in which we are going to develop our speculations, we remind that the $\hat{M}$, $\hat{P}$ operators, defining the Appèl family, are defined by
	\begin{equation}\label{MP}
	\hat{M}=x+\dfrac{A'(\sigma)}{A(\sigma)}\mid_{\sigma=\partial_x}, \qquad\qquad\qquad \hat{P}=\partial_x\;,
	\end{equation}
	where $A(\sigma)$ is an analytic function. \\
	
	According to our introductory remarks, the explicit form of the Appèl polynomials is obtained from the identity (property $iii)$ of Eq. \eqref{p0})
	\begin{equation}\label{an}
	a_n(x)=\left(x+ \dfrac{A'(\partial_x)}{A(\partial_x)}\right)^n 1. 
	\end{equation}
	The use of standard operational rules allows to cast Eq. \eqref{an} in a more convenient form. 
	
	\begin{cor}
	We note indeed that \cite{Sh,Zhu}
	\begin{equation}\label{key}
	a_n(x)=\left(x+ \dfrac{A'(\partial_x)}{A(\partial_x)}\right)\left(x+ \dfrac{A'(\partial_x)}{A(\partial_x)}\right)^{n-1} 1
	\end{equation}
	and noting that
	\begin{equation}\label{key}
	x+ \dfrac{A'(\partial_x)}{A(\partial_x)}=A(\partial_x)x\left( A(\partial_x)\right) ^{-1},
	\end{equation}
	we can write, by iteration 
	\begin{equation}\label{apart}
	a_n(x)=\left( A(\partial_x)x\left( A(\partial_x)\right) ^{-1}\right)^n =A(\partial_x)x^n
	\end{equation}
	Acording to Eq. \eqref{apart}, the generating function of Appèl polynomials reads 
	\begin{equation}\label{sap}
	\sum_{n=0}^\infty \dfrac{t^n}{n!}a_n(x)=A(\partial_x)e^{tx}=A(t)e^{tx}.
	\end{equation}
	It is evident that it consists of two contributions: the exponential term and $A(t)$, which will be defined as the “amplitude”.
    	\end{cor}
	
\begin{cor}
In the case of the two variable Hermite polynomials ($HP$), we have that the amplitude is specified by
	\begin{equation}\label{Ae}
	A(t)=e^{yt^2}
	\end{equation}
	with the multiplicative operator being explicitly defined by
	\begin{equation}\label{key}
	\hat{M}=x+2y\partial_x\;.
	\end{equation}
	The associated polynomial family is, accordingly, provided by
\cite{SLicciardi}
	\begin{equation}\label{Her}
	H_n(x,y)=(x+2y\partial_x)^n 1.
	\end{equation}
	The use of the Crofton identity \cite{Khan}
	\begin{equation}\label{key}
	e^{y\partial_x^m}f(x)=f\left( x+m\;y\;\partial_x^{m-1}\right) e^{y\partial_x^m}
	\end{equation}
	(or of identities \eqref{apart}-\eqref{sap} as well) allows to cast Eq. \eqref{Her} in the form
	\begin{equation}\label{Hexp}
	H_n(x,y)=e^{y\partial_x^2}x^n.
	\end{equation}
	The expansion of the exponential operator in Eq. \eqref{Hexp}, along with the relevant action on the monomial $x^n$, yields the explicit form of the two variable Hermite polynomials \cite{Babusci,Appel}, namely
	\begin{equation}\label{H2}
	H_n(x,y)=n!\sum_{r=0}^{\lfloor \frac{n}{2}\rfloor}\dfrac{x^{n-2r}y^r}{(n-2r)!r!}.
	\end{equation}
\end{cor}
	The operational identity in Eq. \eqref{Hexp} is particularly pregnant from the mathematical point of view. It states that the two variable Hermite \eqref{H2} are solutions of the heat equation and can be used as a pivotal tool to prove the orthogonal properties of this polynomial family \cite{Babusci,SLicciardi}.\\
	
	In this article we consider the polynomial family generated by 
	\begin{equation}\label{newA}
	A(p)=\left( 1+\dfrac{y}{N}p^2\right)^N , \qquad\forall N\in\mathbb{N} ,
	\end{equation}
	study the relevant properties and look at the possibility of defining an associated orthogonal set. 
	
	\section{Quasi-Hermite and Appéll Sequences}
	
	In this section we exploit the general properties of the Appéll polynomials, discussed in the introductory remarks, to state the properties of the associated polynomials.
	
	\begin{defn}
	 Appéll polynomials with amplitude \eqref{newA}, are explicitly defined\footnote{Definition comes according to the discussion of the previous section. It should be noted that\\ $\lim\limits_{N\rightarrow\infty}H_n(x,y;N)=e^{y\partial_x^2}x^n=H_n(x,y)$.} through the identity
	\begin{equation}\label{Hoper}
	H_n(x,y;N)=\left( 1+\dfrac{y}{N}\partial_x^2\right)^N x^n 
	\end{equation}
    and they will be called \textit{Quasi-Hermite-Polynomials} ($QHP$) .
	\end{defn}
	
	\begin{propert}\label{Pr2}
	The relevant recurrences of $QHP$ are obtained after noting that, for this specific case, we get 
\begin{equation}\label{key}
\dfrac{A'(\partial_x)}{A(\partial_x)}=\dfrac{2\;y\;\partial_x}{\left( 1+\dfrac{y}{N}\partial_x^2\right)} ,
\end{equation}
so
	\begin{equation}\label{HN}
	\begin{split}
	& 1) \; \partial_x H_n(x,y;N)=n\left( 1+\dfrac{y}{N}\partial_x^2\right)^N x^{n-1}=nH_{n-1}(x,y;N),\\
	& 2) \; H_{n+1}(x,y;N)=\left( x+\dfrac{2\;y\;\partial_x}{1+\frac{y}{N}\partial_x^2}\right)H_n(x,y;N) ,\\
	& 3) \; H_{n+1}(x,y;N)-xH_n(x,y;N)-2nyH_{n-1}(x,y;N)=\\
	& =\dfrac{y}{N}n(n-1)\left( xH_{n-2}(x,y;N)-H_{n-1}(x,y;N)\right) .
	\end{split}
	\end{equation}
		\end{propert}
	
	\begin{proof}
Properties $1)$ and $2)$ are obtained from the realization of the derivative and multiplicative operators given in Eqs. \eqref{MP}. About the third one, it is the result of some algebraic\footnote{We simplify the writing for brevity by omitting the arguments of the Hermite's.} steps:
\begin{itemize}
	\item [i)] From property $2)$ we write
	\begin{equation*}\label{key}
	\left( 1+\dfrac{y}{N}\partial_x^2\right)H_{n+1}= \left( 1+\dfrac{y}{N}\partial_x^2\right)xH_n+2y\partial_xH_n 
	\end{equation*}
	which provides, from property $1)$,
	\item [ii)] 
	\begin{equation*}\label{key}
	H_{n+1}+\dfrac{y}{N}n(n+1)H_{n-1}=xH_n+\dfrac{y}{N}\left(2nH_{n-1}+n(n-1)xH_{n-2} \right) +2nyH_{n-1} 
	\end{equation*}
	and finally
	\item [iii)] 
	\begin{equation*}
	\begin{split}
	H_{n+1}-xH_n-2nyH_{n-1} &=\dfrac{y}{N}n\left(\left( 2-(n+1)\right)H_{n-1}+(n-1)xH_{n-2}  \right) =\\
	& =\dfrac{y}{N}n(n-1)\left( xH_{n-2}-H_{n-1}\right) .
		\end{split}
	\end{equation*}
\end{itemize}
	\end{proof}	
		
	\begin{prop}
	The explicit form of the $QHP$ is inferred from Eq. \eqref{Hoper}, which yields
	\begin{equation}\label{key}
a) \;	H_{n}(x,y;N)=\sum_{r=0}^{\min\left[N,\;\lfloor \frac{n}{2}\rfloor \right] } \binom{N}{r}\left(\dfrac{y}{N} \right)^r\dfrac{n!}{(n-2r)!}x^{n-2r}, \quad \forall x,y\in\mathbb{R}, \forall n,N\in\mathbb{N}
	\end{equation}
	and the relevant differential equation is 
	\begin{equation}\label{eqd}
	b) \; \left( x+\dfrac{2\;y\;\partial_x}{1+\frac{y}{N}\partial_x^2}\right)\partial_x H_n(x,y;N)=nH_n(x,y;N) .
	\end{equation}
		\end{prop}
	
	\begin{proof}
		a) $\forall x,y\in\mathbb{R}, \forall n,N\in\mathbb{N}$, we use binomial Newton to write
		\begin{equation*}\label{key}
	H_{n}(x,y;N)=\sum_{r=0}^N \binom{N}{r}\left(\dfrac{y}{N} \right)^r \partial_x^{2r}x^n=\sum_{r=0}^{\min\left[N,\;\lfloor \frac{n}{2}\rfloor \right] } \binom{N}{r}\left(\dfrac{y}{N} \right)^r\dfrac{n!}{(n-2r)!}x^{n-2r}	.
		\end{equation*}	
		b) 	The relevant differential equation is easily obtained by applying Eqs. \eqref{HN} in Properties \ref{Pr2}.
	\end{proof}	

\begin{cor}\label{c3}
	After a few algebraic manipulations, Eq. \eqref{eqd} can be reduced to the following third order $ODE$
	\begin{equation}\label{key}
	\dfrac{y}{N}xz^{'''}+y\left( 2-\dfrac{n-2}{N}\right)z^{''}+xz'=nz,\qquad\qquad z=H_n(x,y;N) 
	\end{equation}
which, evidently, tends to the ordinary (two variables) Hermite equation, for large $N$ values.
\end{cor}
	
	\begin{proof}
		By starting from Eq. \eqref{eqd} we write
		\begin{equation*}
		\begin{split}
		& \left( \left( 1+\dfrac{y}{N}\partial_x^2\right)x+2y\partial_x \right)\partial_x z= \left( 1+\dfrac{y}{N}\partial_x^2\right)nz \quad \rightarrow\\
		\rightarrow\; &\; x \partial_x z+\dfrac{y}{N}\partial_x^2x\partial_xz+2y\partial_x^2z=nz+\dfrac{y}{N}\partial_x^2nz  \quad \rightarrow\\
		\rightarrow\; &\; xz'+\dfrac{y}{N}\partial_x(z'+xz'')+2yz''-\dfrac{y}{N}nz''=nz \quad \rightarrow\\
		\rightarrow\; &\; \dfrac{y}{N}xz^{'''}+y\left( 2-\dfrac{n-2}{N}\right)z^{''}+xz'=nz
		\end{split}
		\end{equation*}
	\end{proof}
	
	\begin{cor}
	The $PDE$ satisfied by the $QHP$ (expected to be an extension of the heat equation) is obtained by keeping the partial derivative with respect to y of both sides of Eq. \eqref{Hoper}, namely 
	\begin{equation}\label{pde}
	\partial_y H_n(x,y;N)=\partial_x^2\left( 1+\dfrac{y}{N}\partial_x^2\right)^{N-1} x^n. 
	\end{equation}
\end{cor}

\begin{exmp}
	Eq. \eqref{pde} can eventually be written as
	\begin{equation}\label{key}
\left\lbrace 	\begin{array}{l}
	\partial_y H_n(x,y;N)=\dfrac{\partial_x^2}{1+\frac{y}{N}\partial_x^2}H_n(x,y;N)\\[1.1ex]
	H_n(x,0;N)=x^n
	\end{array}\right. 
	\end{equation}
	The relevant (formal) solution can be obtained as 
	\begin{equation}\label{Uop}
	H_n(x,y;N)=\hat{U}_{y,N} \;x^n, \qquad\qquad \hat{U}_{y,N}=\exp\left\lbrace \int_0^y \dfrac{\partial_x^2}{1+\frac{\xi}{N}\partial_x^2}\;d\xi\right\rbrace ,
	\end{equation}
	where $\hat{U}$ is a kind of evolution operator. To be eventually written as in Eq. \eqref{Hoper}, after explicitly working out the integral in the exponent of Eq. \eqref{Uop}, we find
	\begin{equation}\label{key}
	\hat{U}_{y,N}=\exp\left\lbrace N \log \left(1+ \dfrac{y}{N}\partial_x^2\right) \right\rbrace= \left(1+ \dfrac{y}{N}\partial_x^2\right)^N.
	\end{equation}
	According to the previous definition, the $QHP$ satisfies the composition rule 
	\begin{equation}\label{compProp}
	\hat{U}_{y,N}\hat{U}_{z,N}=\left( 1+\dfrac{y+z}{N}\partial_x^2+\dfrac{yz}{N^2}\partial_x^4\right)^N. 
	\end{equation}
	Therefore, unlike the two variables $HP$ specified by an amplitude that is an exponential , the composition property $\hat{U}_{y,N}\hat{U}_{z,N}=\hat{U}_{y+z,N}$ does not hold, therefore
	\begin{equation}\label{key}
	\hat{U}_{y,N}\hat{U}_{z,N}\neq \hat{U}_{y+z;N}.
	\end{equation}
     An important (albeit naïve) consequence of Eq. \eqref{compProp} is the following composition rule
	\begin{equation}\label{Umeno}
	\hat{U}_{-y,N}\hat{U}_{y,N}x^n=\left(1-\dfrac{y^2}{N^2}\partial_x^4 \right)^N x^n, 
	\end{equation}
	which suggests the necessity of a suitable extension of $QHP$, possibly involving higher order forms, as discussed in the forthcoming section.
\end{exmp}
	
	\begin{obs}
	The non exponential nature of the $QHP$ amplitude determines the further worth to be noted consequence 
	\begin{equation}\label{Lapl}
	\hat{U}_{-y,N}\neq\hat{U}_{y,N}^{-1}=\dfrac{1}{\Gamma(N)}\int_0^\infty s^{N-1}e^{-s\left(1+\frac{y}{N}\partial_x^2 \right) }ds,
	\end{equation}
	where the r.h.s. has been obtained after exploiting standard Laplace transform methods. 
		\end{obs}
	
	We will see in the following that Eq. \eqref{Lapl} is of pivotal importance for the definition of the orthogonal properties of the $QHP$.
	
	\begin{obs}
	Before closing this section, we notice that Eq. \eqref{eqd} can be generalized $\forall m\in\mathbb{N}$ such that 
		\begin{equation}\label{eqdm}
	 \left( x+\dfrac{m\;y\;\partial^{m-1}_x}{1+\frac{y}{N}\partial_x^m}\right)\partial_x H^{(m)}_n(x,y;N)=nH^{(m)}_n(x,y;N) .
	\end{equation}
	and, by following the same procedure provided in the Corollary \ref{c3}, it is possible to deduce the relative differential equantion.
	\end{obs}
	
	\section{Multivariable $QHP$}
	
	Higher order Hermite polynomials (also called \textit{Lacunary HP}) are defined through the operational rule \cite{Babusci,Khan}
	\begin{equation}\label{key}
	H_n^{(m)}(x,y)=e^{y\partial_x^m}x^n, \qquad\quad  \forall m\in\mathbb{N}
	\end{equation}
	and, in analogy, the Higher order $QHP$ are specified by\footnote{We should adopt for the polynomials $H_n(x,y;N)$ the notation $H^{(2)}_n(x,y;N)$, we drop however the superscript for $m=2$ and add it whenever ambiguities arise.} 
	\begin{equation}\label{key}
	H^{(m)}_n(x,y;N)=\hat{U}^{(m)}_{y,N}\;x^n,\qquad\qquad \hat{U}^{(m)}_{y,N}=\left(1+\dfrac{y}{N}\partial_x^m \right)^N. 
	\end{equation}
	
	\begin{exmp}
	According to Eq. \eqref{Umeno}, we find
	\begin{equation}\label{key}
    \hat{U}_{-y,N}\hat{U}_{y,N}\;x^n=\hat{U}^{(4)}_{-y^2,N}\;x^n=
    H_n^{(4)}\left(x,-y^2;N \right)
    \end{equation}
	and, more in general, 
	\begin{equation}\label{key}
	\hat{U}^{(m)}_{-y,N}\hat{U}^{(m)}_{y,N}x^n=\hat{U}^{(2m)}_{-y^2,N}\;x^n=
	H_n^{(2m)}\left(x,-y^2;N \right)
	\end{equation}
		\end{exmp}
	
	\begin{exmp}
	Before going further, we consider the definition of the $QHP$ of order one, which will be referred as \textit{Quasi Binomial Polynomials} ($QBP$), namely
	\begin{equation}\label{key}
	H_n^{(1)}(x,y;N)=\left(1+\dfrac{y}{N}\partial_x \right)^N x^n .
	\end{equation}
	For large $N$ they reduce to $(x+y)^n$ hence the name.
The explicit form of this family of polynomials, writes
	\begin{equation}\label{key}
	H_n^{(1)}(x,y;N)=\sum_{r=0}^N \binom{N}{r}\left(\dfrac{y}{N} \right)^r \partial_x^r x^n= \sum_{r=0}^{\min\left[N,\;n\right] } \binom{N}{r}\left(\dfrac{y}{N} \right)^r\dfrac{n!}{(n-r)!}x^{n-r}.
	\end{equation}
	The same strategy adopted in Corollary \ref{c3}, by exploiting Eq. \eqref{eqdm}, yields for the $QBP$ the $ODE$  
	\begin{equation}\label{key}
     \dfrac{y}{N}xz^{''}+\left[ (x+y)-(n-1)\dfrac{y}{N}\right] z'=nz,\qquad\qquad z=H_n^{(1)}(x,y;N)
    \end{equation}
	and the $PDE$
    \begin{equation}\label{pdF}
 	\left\lbrace 	\begin{array}{l}
	\partial_y F(x,y)=\dfrac{\partial_x}{1+\frac{y}{N}\partial_x}F(x,y)\\[1.2ex]
	F(x,0)=x^n
	\end{array} .\right. 
	\end{equation}
	The last identity can also be cast in the integro-differential form 
	\begin{equation}\label{res}
	\partial_y F(x,y)=\partial_x\int_0^\infty e^{-s}F\left(x-\dfrac{y}{N}s,\;y \right)ds
	\end{equation}
	indeed, the Laplace transform provides the integral representation
	\begin{equation}\label{key}
	\dfrac{1}{1+\frac{y}{N}\partial_x}=\int_0^\infty e^{-s\left( 1+\frac{y}{N}\partial_x\right) }ds
	\end{equation}
	which, one inserted in Eq. \eqref{pdF}, yields
	\begin{equation}\label{key}
	\partial_y F(x,y)=\partial_x\int_0^\infty e^{-s}e^{-\frac{ys}{N}\partial_x }ds F(x,y)
	\end{equation}
	and, after exploiting the shift operator identity $e^{a\partial_x}f(x)=f(x+a)$ \cite{Babusci}, we obtain Eq. \eqref{res}.
\end{exmp}
\begin{exmp}
We can combine the various definition given before to introduce three variables $QHP$ as 
	\begin{equation}\label{key}
	\begin{split}
	H_n^{(1,2)}(x,y_1,y_2;N):&=
	\left(\hat{U}_{y_1,N}^{(1)}\;\hat{U}_{y_2,N}^{(2)}  \right) x^n=
	\left(1 +\dfrac{y_1}{N}\partial_x\right)^N\left(1 +\dfrac{y_2}{N}\partial_x^2\right)^Nx^n=\\
	&=
	\left(1 +\dfrac{y_1}{N}\partial_x\right)^N H_n(x,y_2;N)=\\
& =	n!  \sum_{r=0}^{\min\left[N,n \right] } \binom{N}{r}\left(\dfrac{y_1}{N} \right)^{\!r}\dfrac{H_{n-r}(x,y_2;N)}{(n-r)!}.
%
		\end{split}
	\end{equation}
	Further generalizations can easily be obtained. For example the $m$-th variable extension reads
	\begin{equation}\label{key}
   H_n^{(1,\dots ,m)}(x,y_1,\dots,y_m;N)=\left( \prod_{s=1}^n \hat{U}_{y_s,N}^{(s)}\right) x^n.
\end{equation}
	\end{exmp}

	The examples we have just touched on in this section yields an idea of the possible extensions of this family of polynomials, which will be more carefully discussed in a forthcoming research.
	
	\section{Final Comments}
	
	We have already mentioned the possible orthogonal nature of the $QHP$, in this section we address the problem by the use of the techniques developed in Refs. \cite{Mono,Herm}.
	
	\begin{prop}
We assume that a given function $f(x)$ can be expanded in terms of $QHP$, according to the identity\footnote{The reasons of $``-\!\mid y \mid"$ will be clarified below.}
	\begin{equation}\label{f1}
	f(x)=\sum_{n=0}^\infty a_n H_n(x,-\mid y \mid;N)
	\end{equation}
which can be inverted, thus yielding
	\begin{equation}\label{fx}
    \dfrac{1}{\left(1-\dfrac{\mid y\mid}{N}\partial_x^2 \right)^N }\;f(x)=\sum_{n=0}^\infty a_n x^n.
    \end{equation}
	The use of Eq. \eqref{Lapl} allows to cast the l.h.s. of Eq. \eqref{fx} in the form
	\begin{equation}\label{GfintS}
	\dfrac{1}{\Gamma(N)}\int_0^\infty s^{N-1}e^{-s\left( 1-\frac{\mid y \mid}{N}\partial_x^2\right) }ds\;f(x)=\sum_{n=0}^\infty a_n x^n.
	\end{equation}
		\end{prop}
	
	\begin{cor}
    Eq. \eqref{GfintS} can be so further elaborated:
	
	\begin{enumerate}
		\item  We apply the Gauss-Weierstrass transform \cite{SLicciardi} to write
     	    \begin{equation}\label{key}
     	    	 \begin{split}
               e^{s\frac{\mid y \mid}{N}\partial_x^2}f(x)&=\dfrac{1}{2\sqrt{\pi\;s\frac{\mid y \mid}{N}}}\int_{-\infty}^\infty \exp\left\lbrace -\dfrac{(x-\xi)^2}{4\;s\frac{\mid y \mid}{N}}\right\rbrace f(\xi)d\xi=\\
              &  \;=\dfrac{1}{2\sqrt{\pi\;s\frac{\mid y \mid}{N}}}
               	\int_{-\infty}^\infty  e^{-
               		\frac{1}
               		{4\;s\frac{\mid y \mid}{N}}
               		\xi^2} e^{\frac{x}{2s\frac{\mid y \mid}{N}}\xi}e^{-\frac{x^2}{4s\frac{\mid y \mid}{N}}}d\xi.        
               	\end{split}
            \end{equation}
	It holds for $s\frac{\mid y \mid}{N}\geq 0$ (hence the choice of the sign in the expansion \eqref{f1}).\\
	
	\item We use the two variable Hermite generating function\footnote{We have $\sum_{n=0}^\infty \frac{t^n}{n!}H_n(x,y)=e^{xt+yt^2}$.} \cite{Babusci} to write
	 \begin{equation}\label{gf}
	e^{s\frac{\mid y \mid}{N}\partial_x^2}f(x)=
	\dfrac{1}{2\sqrt{\pi\;s\frac{\mid y \mid}{N}}}	
	\sum_{n=0}^\infty \dfrac{x^n}{n!}
	\int_{-\infty}^\infty 
	H_n\left( \dfrac{\xi}{2s\frac{\mid y \mid}{N}}, - \dfrac{1}{4 s\frac{\mid y \mid}{N}}\right) 
	 e^{-
	 	\frac{\xi^2}
	 	{4\;s\frac{\mid y \mid}{N}}
	 	} f(\xi) d\xi .	
	\end{equation}

	\item We insert the result of Eq. \eqref{gf}, in Eq. \eqref{GfintS} and compare the like $x$ powers, thus eventually finding
	\begin{equation}\label{last}
	\begin{split}
	& a_n=\dfrac{1}{\Gamma(N)n!\;2\sqrt{\pi\dfrac{\mid y\mid}{N}}}\int_0^\infty s^{N-\frac{3}{2}}e^{-s}\; {}_nG_{y,N}(s)ds,\\
	& {}_nG_{y,N}(s)=\int_{-\infty}^\infty 
	H_n\left( \dfrac{\xi}{2s\frac{\mid y \mid}{N}}, - \dfrac{1}{4 s\frac{\mid y \mid}{N}}\right) 
	e^{-
		\frac{\xi^2}
		{4\;s\frac{\mid y \mid}{N}}
	} f(\xi) d\xi .
	\end{split}
	\end{equation}
		\end{enumerate} 
	According to the above results the expansion holds only if the integrals appearing in Eq. \eqref{last} are converging. In order to provide an example we consider the generalization of the Glaisher formula \cite{Glaisher}. Namely Eq. \eqref{fx}, for $f(x)=e^{-x^2}$, becomes 
	\begin{equation}\label{key}
	\begin{split}
	F(x,y;N)&=\dfrac{1}{\left(1-\dfrac{\mid y\mid}{N}\partial_x^2 \right)^N }f(x)=\dfrac{1}{\Gamma(N)}\int_0^\infty s^{N-1}e^{-s\left( 1-\frac{\mid y \mid}{N}\partial_x^2\right) }e^{-x^2}ds=\\
	& = \dfrac{1}{\Gamma(N)}\int_0^\infty s^{N-1}e^{-s} \dfrac{e^{-\frac{x^2}{1+4\frac{\mid y\mid}{N}s}}}{\sqrt{1+4\dfrac{\mid y\mid}{N}s}}\;ds.
		\end{split}
	\end{equation}
	 For very large $N$,\; $\frac{1}{\left(1-\frac{\mid y\mid}{N}\partial_x^2 \right)^N }e^{-x^2}$, reduces to the ordinary Glaisher identity
	\begin{equation}\label{key}
	\lim\limits_{N\rightarrow\infty}F(x,y;N)=e^{\mid y\mid\partial_x^2}e^{-x^2}=\dfrac{1}{\sqrt{1+4y}}e^{-\frac{x^2}{1+4y}}.
	\end{equation}
	In Figs. \ref{Fig1} we have reported $F(x,y;N)$ vs. $x$ for different values of $N$ and $y$.
	
	\begin{figure}[h]
		\centering
		\begin{subfigure}[c]{0.48\textwidth}
			\includegraphics[width=1.\linewidth]{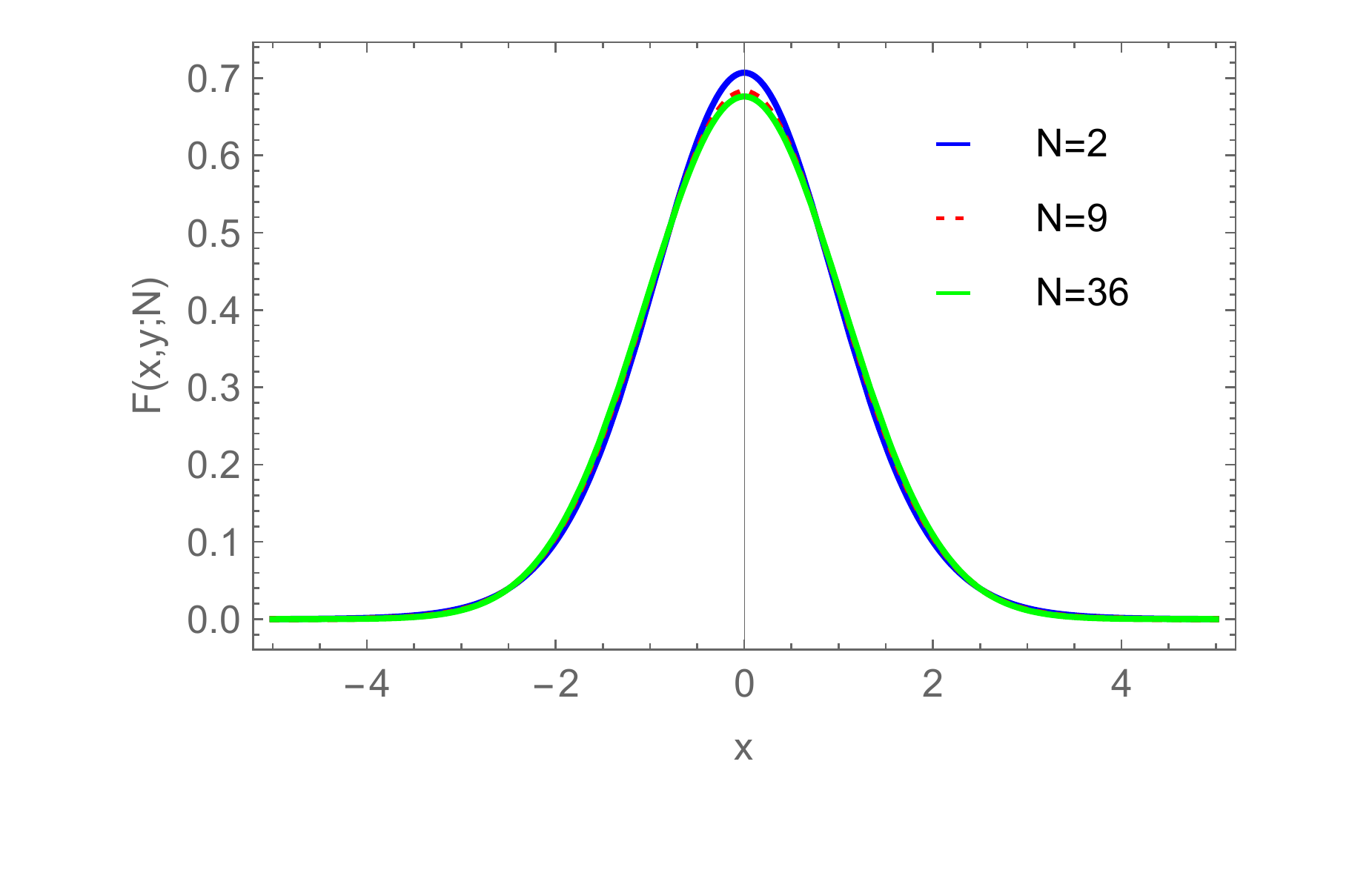}
			\caption{$y=0.3$}
		\end{subfigure}
		\begin{subfigure}[c]{0.48\textwidth}
			\includegraphics[width=1.\linewidth]{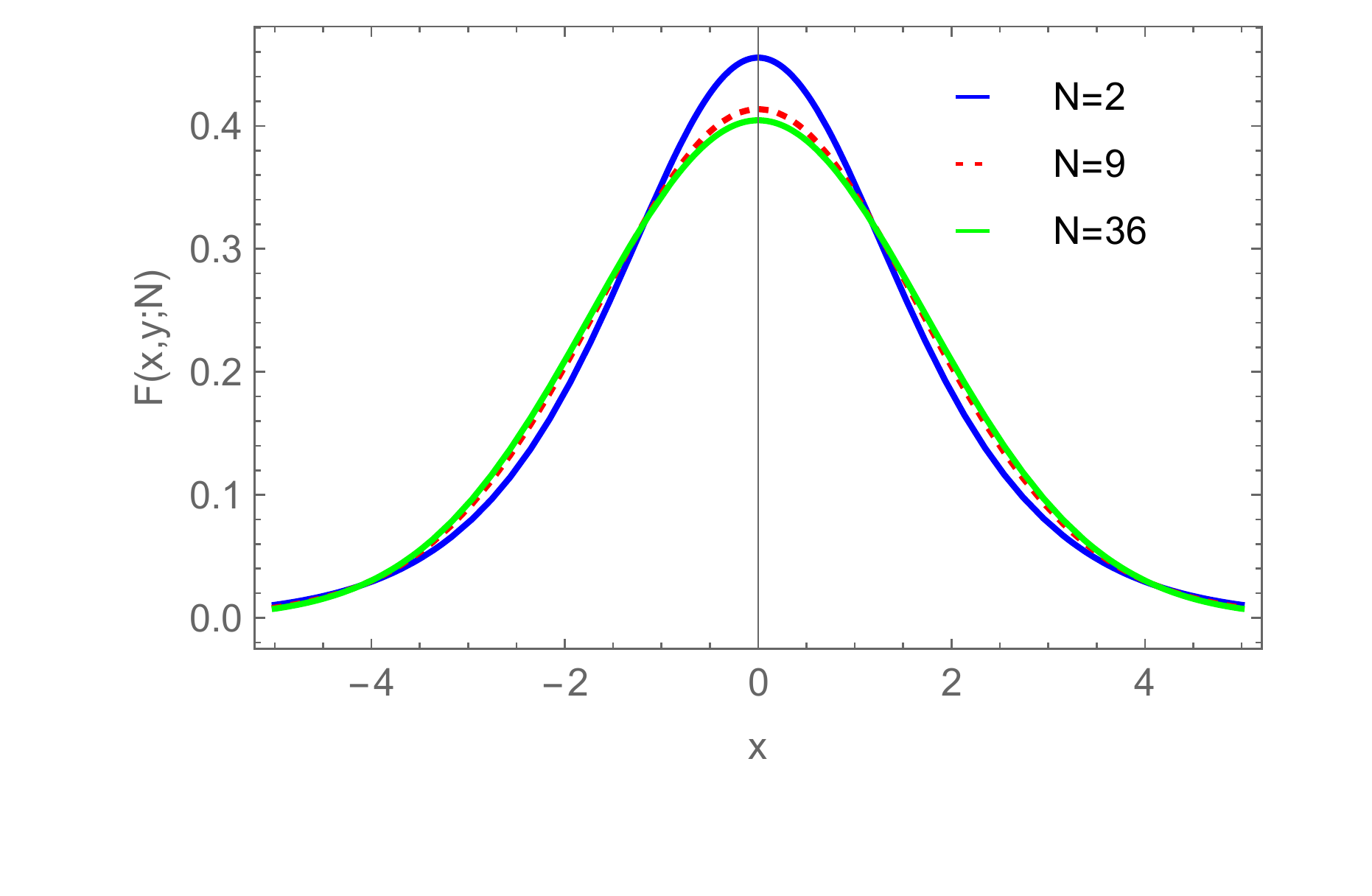}
			\caption{$y=1.3$}
		\end{subfigure}	
		\caption{$F(x,y;N)$ vs. $x$ for different values of $N$ and $y$}
		\label{Fig1}
	\end{figure}
	
\end{cor}
	
	The definition of higher order ($QHP$) is not unique and another possibility is offered by the relation
	\begin{equation}\label{pr}
	H_n^{(q,p)}(x,y,z;N)=\left( 1+\dfrac{y}{N}\partial_x^q+\dfrac{z}{N}\partial_x^p\right)^N x^n 
	\end{equation}
	where $q<p$ are relatively prime integers. The definition in Eq. \eqref{pr} allows to write the composition identity \eqref{compProp} as  
	\begin{equation}\label{key}
	\hat{U}_{y,N}\hat{U}_{z,N}x^n=H_n^{(2,4)}\left(x,y+z,yz;N \right) ^N\;x^n.
	\end{equation}
	In this paper we have gone through different aspects of the theory of Hermite polynomials, which has allowed the introduction of a family of Quasi Hermite polynomials. The relevant properties have been studied with the help of the formalism of Monomiality. \\
	
	In a forthcoming paper we will extend the present analysis to the definition of two variables Quasi Laguerre polynomials.
\\
	
\textbf{Acknowledgements}\\
 The work of Dr. S. Licciardi was supported by an Enea Research Center individual fellowship and under the auspices of INDAM’s GNFM (Italy).\\

%

	{}
	
\end{document}